\documentclass[11pt,final]{amsart}
\usepackage{amssymb, amsmath, amsthm}
\usepackage{mathrsfs}  
\usepackage{graphicx}
\usepackage[colorlinks=true, citecolor=blue, linkcolor=red, urlcolor=blue]{hyperref}
\usepackage[text={5.5in,8in},centering]{geometry}
 \usepackage{xcolor}
 \usepackage{tcolorbox}

\usepackage{mathrsfs}  

\usepackage{bm}
\usepackage{ifdraft}
\ifoptionfinal{
\usepackage[disable]{todonotes}
}{
\usepackage[norefs, nocites]{refcheck}

\usepackage[notref, notcite]{showkeys}
\usepackage[bordercolor=white, color=white]{todonotes}
}

\makeatletter\providecommand\@dotsep{5}\def\listtodoname{List of Todos}\def\listoftodos{\hypersetup{linkcolor=black}\@starttoc{tdo}\listtodoname\hypersetup{linkcolor=blue}}\makeatother
\usepackage{hyperref}
\usepackage{etoolbox}

\newtheorem{theorem}{Theorem}[]

\newtheorem{corollary}[theorem]{Corollary}

\theoremstyle{remark}

\numberwithin{equation}{section}

\def\C{\mathbb C}
\def\R{\mathbb R}

\def\N{\mathbb N}

\renewcommand{\leq}{\leqslant}
\renewcommand{\geq}{\geqslant}

\def\p{\partial}

\newcommand{\pair}[1]{\left\langle #1 \right\rangle}

\newcommand*\xbar[1]{%
   \hbox{%
     \vbox{%
       \hrule height 0.5pt % The actual bar
       \kern0.5ex%         % Distance between bar and symbol
       \hbox{%
         \ensuremath{#1}%
       }%
     }%
   }%
} 

\title[Remarks on the anisotropic Calder\'{o}n problem]{Remarks on the anisotropic Calder\'{o}n problem}
\author[C\^{a}rstea]{C\u{a}t\u{a}lin I. C\^{a}rstea}
\address{School of Mathematics, Sichuan University, Chengdu, Sichuan, China.} 
\email{catalin.carstea@gmail.com}

\author[Feizmohammadi]{Ali Feizmohammadi}
\address{Fields institute, 222 College St, Toronto, ON M5T 3J1}
\email{afeizmoh@fields.utoronto.ca}

\author[Oksanen]{Lauri Oksanen}
\address{Department of Mathematics and Statistics, University of Helsinki, P.O 68, University of Helsinki, Finland.}
\email{lauri.oksanen@helsinki.fi}

\begin{document}

%\begin{titlepage}
%\maketitle
%\end{titlepage}

\maketitle
\begin{abstract}
We show uniqueness results for the anisotropic Calder\'{o}n problem stated on transversally anisotropic manifolds. Moreover, we give a convexity result for the range of Dirichlet-to-Neumann maps on general Riemannian manifolds near the zero potential. Finally, we present results for Calder\'{o}n type inverse problems associated to semilinear elliptic equations on general Riemannian manifolds.
\end{abstract}
\section{Introduction}
In \cite{Calderon}, A. P. Calder\'{o}n posed the following problem: Is it possible to determine the conductivity matrix of a medium by performing electrical measurements on its boundary? This problem is commonly known as the anisotropic Calder\'{o}n problem and it has a geometric formulation, see e.g. \cite{LU}. To state this formulation, let $(M,g)$ be a smooth compact Riemannian manifold with boundary. We write $M^{\textrm{int}}$ for the interior of $M$ and $\p M$ for its boundary so that $M=M^{\textrm{int}}\cup \p M$. 
The dimension of $M$ is denoted by $n$ and we assume that $n \geq 2$.
The Laplace--Beltrami operator $\Delta_g$ on $(M,g)$ is defined in local coordinates by
\begin{equation}
\label{laplace}
\Delta_gu=\frac{1}{\sqrt{\det g}}\sum_{j,k=1}^n\frac{\p}{\p x^j}\left(\sqrt{\det g}\,g^{jk}\frac{\p u}{\p x^k}\right)\quad \forall\, u\in C^{\infty}(M).
\end{equation} 
Let $f\in H^{\frac{3}{2}}(\p M)$ and consider the unique harmonic function with Dirichlet data $f$, that is to say,
\begin{equation}\label{pf}
\begin{cases}
-\Delta_{g} u=0 & \mbox{on}\ M^{\textrm {int}},\\  
u =f & \mbox{on}\ \p M.
\end{cases}
\end{equation}
The Dirichlet-to-Neumann map for \eqref{pf} is defined by
$$ \Lambda_{g}(f)=\p_\nu u|_{\p M},$$
where $\nu$ stands for the exterior unit normal vector field on $\p M$ and $u\in H^2(M)$ is the solution to \eqref{pf}. The geometric Calder\'{o}n problem can now be stated as follows: 

\begin{itemize}
	\item [(IP1)]{Does the knowledge of the Dirichlet-to-Neumann map $\Lambda_g$ uniquely determine the metric $g$? }
\end{itemize}
As noted by Luc Tartar (account given in \cite{KV}), there is a natural obstruction to uniqueness that arises from isometries. Namely, if $\Phi:M\to M$ is a diffeomorphism that fixes the boundary, i.e. $\Phi|_{\p M}=\textrm{Id}$, then there holds, $\Lambda_{\Phi^\star g}=\Lambda_{g}$. In dimension two, an additional obstruction arises due to the conformal invariance of the Laplace--Beltrami operator. The geometric Calder\'{o}n problem can thus be re-stated as the question of injectivity of the map $\Lambda_g$, up to the natural obstructions discussed above.

When $n=2$, the geometric Calder\'{o}n problem has been solved, see \cite{Na96} for the case of Riemannian metrics on bounded domains in the plane and \cite{LaUh} for the case of Riemannian surfaces. 
In the review of previous results below, we will restrict our attention to dimensions $n\geq 3$, where the problem has remained open to large extent. 

The works \cite{LLS,LTU,LaUh,LU} study the Calder\'{o}n problem on real-analytic manifolds with a real-analytic metric and an analogous result is proved in \cite{GS} for Einstein manifolds (which are real-analytic in their interior). 
Outside the category of real-analytic metrics, results are sparse and they are all related to a simpler variation of the anisotropic Calder\'{o}n problem where one assumes that the conformal class of the manifold is a priori known. In this case, the question becomes as follows:
\begin{itemize}
	\item [(IP2)]{If $\Lambda_{cg}=\Lambda_{g}$ for some smooth positive function $c$ on $M$, does it follow that $c=1$?}
\end{itemize} 
Before reviewing the results on (IP2), we mention a closely related inverse problem associated to the Schr\"{o}dinger operator on Riemannian manifolds. To this end, let $V \in L^{\infty}(M)$ and consider the equation
 \begin{equation}\label{pf_schrodinger}
\begin{cases}
(-\Delta_{g}+V) u=0 & \mbox{on}\ M^{\textrm {int}},\\  
 u =f & \mbox{on}\ \p M.
\end{cases} 
 \end{equation}
We assume that zero is not a Dirichlet eigenvalue of the operator $-\Delta_{g}+V$ and define the Dirichlet-to-Neumann map
$$\Lambda_{g,V}(f)=\p_\nu u|_{\p M},$$
where $u\in H^2(M)$ is the unique solution to \eqref{pf_schrodinger} with Dirichlet data $f\in H^{\frac{3}{2}}(\p M)$. The inverse problem reads:
	\begin{itemize}
		\item [(IP3)]{If $\Lambda_{g,V_1}=\Lambda_{g,V_2}$ for $V_1,V_2\in L^{\infty}(M)$, does it follow that $V_1=V_2$?}
	\end{itemize} 
It is well known that an affirmitive answer to (IP3) implies (IP2). This is based on the observation that
$$ -\Delta_{cg} u= c^{-\frac{n+2}{4}}(-\Delta_g +V)(c^{\frac{n-2}{4}}u)\quad \text{with} \quad V=c^{-\frac{n-2}{4}}\Delta_g(c^{\frac{n-2}{4}}).$$

The seminal work \cite{SU} solves both the problems (IP2) and (IP3) when $M$ is a domain in $\R^n$ and $g$ is the Euclidean metric. This is achieved via introduction of complex geometric optics solutions to Schr\"{o}dinger equations. The works \cite{DKSU,DKLS} generalize the idea of complex geometric optics further by studying (IP2)--(IP3) on the so-called {\em conformally transversally anisotropic} manifolds (CTA in short), namely manifolds that are embedded in a cylinder up to a conformal rescaling. That is, 
\begin{equation}\label{cta_def} 
M \subset \R\times M_0, \qquad g(t,x)= c(t,x) (dt^2+g_0(x)),
\end{equation}
where the transversal factor $(M_0,g_0)$ is a smooth Riemannian manifold with boundary. Let us also mention that a Riemannian manifold $(M,g)$ is called {\em transversally anisotropic} if it satisfies \eqref{cta_def} with $c=1$. The works \cite{DKSU,DKLS} solve (IP2)--(IP3) on CTA manifolds under the additional assumption that the geodesic ray transform on $(M_0,g_0)$ is injective. A constructive counterpart to the uniqueness proofs in \cite{DKSU,DKLS} is given in \cite{FKOU} for continuous potentials $V$. We also mention the works \cite{DKLLS,KLS} that study a linearized version of (IP2)--(IP3) on transversally anisotropic manifolds (recall that the Dirichlet-to-Neumann map $\Lambda_{g,V}$ depends nonlinearly on $V$), under the assumption that $(M_0,g_0)$ is real-analytic. Outside the category of CTA manifolds, we refer the reader to \cite{UW21} that studies a variation of the anisotropic Calder\'{o}n problem associated to the operator $-\Delta_g +V-\lambda^2$ for a fixed sufficiently large parameter $|\lambda|\gg\|V\|_{L^{\infty}(M)}$. Finally, we refer the reader to \cite{Gu} for a survey of the anisotropic Calder\'{o}n problem.

\subsection{Main results} Our first result concerns (IP2) on transversally anisotropic manifolds. 

\begin{theorem}
	\label{thm_ta_1}
	Let $(M,g)$ be a transversally anisotropic manifold of dimension $n\geq 3$ and assume that the transversal factor $(M_0,g_0)$ is nontrapping. Suppose that $c>0$ is a smooth function on $M$. If
	$$\Lambda_{cg} = \Lambda_g,$$  
	then $c=1$ on $M$.
\end{theorem}

We now turn to (IP3). We show a rigidity result on transversally anisotropic manifolds at the zero potential.

\begin{theorem}
	\label{thm_ta_2}
	Let $(M,g)$ be a transversally anisotropic manifold of dimension $n\geq 3$ and assume that the transversal factor $(M_0,g_0)$ is nontrapping. Let $\delta\in (0,1)$ be small, depending only on $(M,g)$, and assume that $V$ is real valued and satisfies $\|V\|_{L^{\infty}(M)}\leq \delta$. If
	$$\Lambda_{g,V} = \Lambda_{g,0},$$ 
	then $V=0$ on $M$.
\end{theorem}

Next, we discuss a convexity result for the range of the Dirichlet-to-Neumann maps associated to Schr\"{o}dinger equations of the form \eqref{pf_schrodinger} on a general Riemannian manifold, assuming that the potential $V$ is sufficiently small.  We prove the following theorem.

\begin{theorem}
	\label{thm_convex}
	Let $(M,g)$ be a smooth compact Riemannian manifold of dimension $n\geq 2$ with boundary. Let $\delta\in (0,1)$ be small depending only on $(M,g)$. For $j=1,2$, let $V_j\in L^{\infty}(M)$ be real valued and satisfy $\|V_j\|_{L^{\infty}(M)}\leq \delta$. There holds,
	$$ \Lambda_{g,(1-t)V_1+ tV_2} \geq (1-t)\,\Lambda_{g,V_1}+t\,\Lambda_{g,V_2}\quad \forall\, t\in [0,1].$$
 Moreover, for any $t \in (0,1)$, the above inequality is strict  unless $V_1=V_2$ on $M$.
\end{theorem}

The inequality in the theorem is in the sense of positive semidefinite operators, that is,
\begin{equation}
\label{claim_convexity}
\int_{\p M} f\Lambda_{g,(1-t)V_1+tV_2}(f)\,dV_g\geq (1-t)\int_{\p M} f\Lambda_{g,V_1}(f)\,dV_g+t\int_{\p M} f\Lambda_{g,V_2}(f)\,dV_g,
\end{equation}
for all $t\in [0,1]$ and all $f\in H^{\frac{3}{2}}(\p M)$.

As a consequence of the theorem the following deformation rigidity holds.

\begin{corollary}
	Let $(M,g)$ be a smooth compact Riemannian manifold of dimension $n\geq 2$ with boundary. Let $V\in L^{\infty}(M)$ be real valued. If $\Lambda_{g,tV} = \Lambda_{g,0}$ for all small $t>0$, then $V = 0$ on $M$.
\end{corollary}

We also derive some results for inverse problems associated to  semilinear elliptic equations on general Riemannian manifolds. This is an area of research that has received considerable attention lately. To introduce this class of inverse problems in a general framework, let $F\in C^{\infty}(M\times \R)$ satisfy 
\begin{itemize}
\item[(A1)]{$F(x,t_k)=0$ for some sequence $t_k \in \R$, $k\in J$},
\item[(A2)]{$0$ is not a Dirichlet eigenvalue of $-\Delta_g +\p_s F(x,s)|_{s=t_k}$, for any $k\in J$,}
\end{itemize}
where $J$ is a countable index set that could possibly have a single, or finite or infinite elements. We consider the semilinear equation
\begin{equation}\label{pf_nonlinear}
\begin{cases}
-\Delta_{g} u+F(x,u)=0 & \mbox{on}\ M^{\textrm {int}},\\  
u =t_k+f & \mbox{on}\ \p M.
\end{cases}
\end{equation}
It is well known, see for e.g. \cite{FO20,LLLS21a} that subject to the conditions (A1)--(A2), the above semilinear elliptic PDE is well posed in the sense that given each $f$ sufficiently close to zero, there is a unique solution which is close to $t_k$. Precisely, given any $\alpha \in (0,1)$ and $k \in J$, there exists $\delta_{k,0}, \delta_{k,1}>0$ such that if 
\begin{equation}
\label{f_assumption}
f\in C^{2,\alpha}(\p M)\quad \text{satisfies}\quad  \|f\|_{C^{2,\alpha}(\p M)}\leq \delta_{k,0},
\end{equation}
then there exists a unique solution $u\in C^{2,\alpha}(M)$ to \eqref{pf_nonlinear} under the additional imposition that
\begin{equation}\label{u_impose}
\|u-t_k\|_{C^{2,\alpha}(M)}\leq \delta_{k,1}.\end{equation}
Moreover, the dependence of $u$ on $f$ is continuous. We may thus define a Dirichlet-to-Neumann map associated to \eqref{pf_nonlinear} by
\begin{equation}\label{DN_nonlinear} 
\Lambda_{g,F}(t_k+f) = \p_\nu u|_{\p M},\end{equation}
where it is to be understood that $k\in J$, $f$ satisfies \eqref{f_assumption} and that $u$ is the unique solution to \eqref{pf_nonlinear} satisfying \eqref{u_impose}. We may now formulate an inverse problem associated to the semilinear equation \eqref{pf_nonlinear}.
	\begin{itemize}
		\item [(IP4)]{For $j=1,2$, let $(M,g_j)$ be a smooth compact Riemannian manifold with boundary. Let $F_1,F_2\in C^{\infty}(M\times \R)$ satisfy (A1)--(A2). If $$\Lambda_{g_1,F_1}(t_k+f)=\Lambda_{g_2,F_2}(t_k+f),$$ 
		for all $k\in J$ and all $f$ satisfying \eqref{f_assumption}, does it follow that $F_1=F_2$ on $M\times \R$ and that there exists a diffeomorphism $\Phi:M\to M$ that fixes the boundary and satisfies $\Phi^\star g_2=g_1$ on $M$?}
	\end{itemize} 

In the case that $g_1 = g_2 = g$ and $(M,g)$ is an Euclidean domain,
this inverse problem has been actively studied over the last couple of decades, see for example \cite{IN95,IS94,Sun04,Sun10,SU97}. Recently, based on the higher order linearization method of \cite{KLU}, introduced in the context of nonlinear hyperbolic equations, new results have appeared also in the case that the manifolds $(M,g_j)$, $j=1,2$, are conformally transversally anisotropic, see \cite{FO20,KU20b,KU20c,LLLS21a,LLLS21b,LLST} and the references therein. In this work, we present results for general Riemannian manifolds. Our first result regarding (IP4) says that the linear case is rigid among nonlinear perturbations. 

\begin{theorem}
	\label{thm_main_0}
	Let $m\geq 2$ be an integer and let $(M,g)$ be a smooth compact Riemannian manifold of dimension $n\geq 2$ with boundary. Let $F(x,s) = V(x)s^m$ for some real valued $V \in C^{\infty}(M)$ and assume that
	$$\Lambda_{g,F}(f)= \Lambda_{g,0}(f),$$
	for all $f\in C^{2,\alpha}(\p M)$ satisfying \eqref{f_assumption} with $J=\{1\}$ and $t_1=0$. Then $F=0$ on $M\times \R$.
\end{theorem}

Finally, we state a result on recovery of a general Riemannian manifold, assuming a nonlinearity of the form $F(x,s)=s\,G(s)$ for $x\in M$ and $s\in \R$. Here $G$ is a smooth, a priori known function that satisfies 
\begin{itemize}
	\item[(H1)]{There exists a strictly monotone sequence $\{t_kG'(t_k)\}_{k\in\N}\subset \R$ such that $G(t_k)=0$, $t_kG'(t_k)>1$ for all $k\in \N$ and $$\sum_{k=1}^{\infty}\frac{1}{t_k\,G'(t_k)}=\infty.$$}
\end{itemize}
It is straightforward to see that $F$ satisfies (A1)--(A2) in this case. As an example note that $G(s)=\sin(s)$ satisfies (H1). We prove the following theorem.
\begin{theorem}
	\label{thm_main}
	For $j=1,2$, let $(M,g_{j})$ be smooth compact Riemannian manifold of dimension $n\geq 2$ with boundary. Let $F\in C^{\infty}(M\times \R)$ be defined via $F(x,s)=sG(s)$ where $G$ is a smooth function on $\R$  satisfying (H1). If
	$$ \Lambda_{g_1,F}(f)=\Lambda_{g_2,F}(f),$$
	for all $f$ satisfying \eqref{f_assumption} with $J=\N$ and $\{t_k\}_{k\in \N}$ as given by (H1), then there exists a diffeomorphism $\Phi:M\to M$ that fixes the boundary and satisfies
	$$ g_{1}= \Phi^{\star}g_{2}\quad \text{on $M$}.$$ 
\end{theorem}

\section{Proofs of the main theorems}
Throughout this section, we write $\nabla$ for the gradient on $(M,g)$ and, given any $p\in M$ and $X,Y \in T_pM$, let $\pair{X,Y}$ stand for $g(X,Y)$. We begin with the proof of our uniqueness result regarding (IP2).

\begin{proof}[Proof of Theorem~\ref{thm_ta_1}]
	It is well known that $\Lambda_{cg}=\Lambda_g$ implies that $c|_{\p M}=1$ and $\p_\nu c|_{\p M}=0$, see e.g. \cite{DKLS}. In view of the equality 
	$$ -\Delta_{cg} u = c^{-\frac{n+2}{4}}(-\Delta_g + V)(c^{\frac{n-2}{4}}u),$$
	with
	\begin{equation}\label{V_def}
	V = c^{-\frac{n-2}{4}}\Delta_g(c^{\frac{n-2}{4}}).
	\end{equation}
	we obtain from $\Lambda_{cg}$ the Dirichlet-to-Neumann map, $\Lambda_{g,V}$, for the equation
	\begin{align}
	\label{schrodinger}
	- \Delta_{g} u+V u &= 0,\quad \text{on $M$},
	\\\notag
	u|_{\p M} &= f.
	\end{align}
	Moreover, since $\Lambda_{cg}=\Lambda_{g}$, we conclude that the Dirichlet-to-Neumann map for \eqref{schrodinger} is equal to the Dirichlet-to-Neumann map for harmonic functions on $(M,g)$, that is to say 
	$$ \Lambda_{g,V} =\Lambda_{g,0},$$
	where $V$ is given via \eqref{V_def}. Let us extend $V$ to $\R\times M_0$ by setting it to zero outside $M$. Let $\gamma:I\to M_0$ be an inextendible unit speed geodesic on $M_0$. Applying equation (3.6) in \cite{DKLS}, taking $\lambda=0$ there, we deduce that 
	\begin{equation}\label{geodesic_eq}\int_I \hat{V}(0,\gamma(s))\,ds =0,\end{equation}
	where $\hat{V}(s,x)$ is the Fourier transform of $V(t,x)$ in the $t$-variable, that is, 
	$$ \hat{V}(s,x)=\int_\R V(t,x) e^{-ist}\,dt.$$
	As equation \eqref{geodesic_eq} is valid along all inextendible geodesics on $M_0$,  we conclude via lifting $\hat{V}(0,x)$ to the normal bundle and applying Santal\'o's formula, see \cite[Lemma 3.3]{GMT} for a version that does not impose convexity conditions on $\p M$, that 
	$$ \int_{M_0} \hat{V}(0,x)\,dV_{g_0}=0.$$
	Noting that $\hat{V}(0,x)= \int_\R V(t,x)\,dt,$
	we deduce that
	$$\int_M V\,dV_g=0.$$
	Recalling that $V= \frac{\Delta_gw}{w}$ with $w=c^{\frac{n-2}{4}}$, we write
	$$ 0= \int_M \frac{\Delta_g w}{w}\,dV_g =-\int_M \pair{\nabla w,\nabla w^{-1}}\,dV_g=\int_M w^{-2}\,\pair{\nabla w,\nabla w}\,dV_g,$$
	where we used the fact that $\p_\nu w|_{\p M}=0$. Thus, $\pair{\nabla w,\nabla w}=0$ on $M$. As $w=1$ on $\p M$, we conclude that $w=1$ everywhere on $M$.
\end{proof}

Next, we prove our rigidity result regarding (IP3).

\begin{proof}[Proof of Theorem~\ref{thm_ta_2}]
	We start by noting that analogously to the proof of Theorem~\ref{thm_ta_1}, one can use equation (3.6) in \cite{DKLS} together with Santalo's formula to conclude that
	\begin{equation}\label{V_int}
	\int_M V\,dV_g =0.
	\end{equation}
	Next, we write $u$ for the solution to \eqref{pf_schrodinger} with Dirichlet boundary data $f=1$. There holds,
	\begin{align*}
	u = 1 + \dot u + r,
	\end{align*}
	where
	\begin{align}\label{eq_dotu}
	- \Delta_g \dot u + V &= 0,\quad \text{on $M^{\textrm{int}}$},
	\\\notag
	\dot u|_{\p M} &= 0,
	\end{align}
	and
	\begin{align}\label{eq_remainder_1}
	- \Delta_g r + V r+ V\dot u &= 0,\quad \text{on $M^{\textrm{int}}$},
	\\\notag
	r|_{\p M} &= 0.
	\end{align}
	By elliptic estimates, we have 
	\begin{equation}
	\label{est_1}
	\|r\|_{H^2(M)} \leq C\,\delta\, \|\dot u\|_{L^2(M)},
	\end{equation}
	for some constant $C>0$ that depends only on $(M,g)$. Also in view of the Poincar\'{e} inequality, there holds
	\begin{equation}
	\label{poincare_1}
	C_0\int_M |\dot u|^2\,dV_g \leq \int_M |\nabla \dot u|^2\,dV_g,
	\end{equation}
	for some constant $C_0>0$ that only depends on $(M,g)$.
	As $0 = \Lambda_{g,0}(f) = \Lambda_{g,V}(f)$ for the choice $f=1$, there holds
   \begin{align*}
0 
&= 
\int_{\p M} u \p_\nu u \,dV_g
= 
\int_M u \Delta_g u\,dV_g + \int_M \pair{\nabla u, \nabla u}\,dV_g
\\&= \int_M u V u\,dV_g + \int_M \pair{\nabla u, \nabla u}\,dV_g
\\&= 
\int_M V \,dV_g+ 2\int_M V \dot u\,dV_g + 2\int_M V r\,dV_g +2\int_M V\dot u\,r\,dV_g +\int_M V \dot u^2\,dV_g\\
&\,\,+\int_M V r^2 \,dV_g+ \int_M |\nabla r|^2\,dV_g+\int_M |\nabla\dot u|^2\,dV_g +2\int_M \pair{\nabla\dot u, \nabla r}\,dV_g.
    \end{align*}
Using equation \eqref{V_int} together with the identities
    \begin{align*}
0 &= \int_M (-\Delta_g \dot u + V) \dot u\,dV_g
= \int_M \pair{\nabla \dot u,\nabla \dot u}\,dV_g + \int_M V \dot u\,dV_g,
\\
0 &= \int_M (-\Delta_g \dot u + V) r\,dV_g
= \int_M \pair{\nabla \dot u,\nabla r}\,dV_g + \int_M V r\,dV_g,
    \end{align*}
we get
   \begin{align*}
0 
&= -\int_M |\nabla\dot u|^2\,dV_g +2\int_M V\dot u\,r\,dV_g +\int_M V \dot u^2\,dV_g\\
&\,\,+\int_M V r^2\,dV_g + \int_M |\nabla r|^2\,dV_g.
\end{align*}
This can be reduced further by using
    \begin{align*}
0 &= \int_M (-\Delta_g r + V r + V \dot u) r\,dV_g
= \int_M |\nabla r|^2\,dV_g + \int_M V r^2\,dV_g + \int_M V \dot u r\,dV_g.
    \end{align*}
That is,
   \begin{align*}
0 
&= -\int_M |\nabla\dot u|^2\,dV_g +\int_M V\dot u\,r\,dV_g +\int_M V \dot u^2\,dV_g.
\end{align*}
Rearranging, we get 
 \begin{align}
 \label{identity_esp}
\int_M |\nabla\dot u|^2\,dV_g &=  \int_M V\dot u\,r\,dV_g +\int_M V \dot u^2\,dV_g.
\end{align}
The claim follows from combining Poincar\'{e}'s inequality \eqref{poincare_1} and the following bounds
 \begin{align*}
|\int_M V\dot u\,r\,dV_g| &\leq C\delta^2 \|\dot u\|_{L^2(M)}^2,\\
|\int_M V\dot u^2\,dV_g| &\leq \delta \|\dot u\|_{L^2(M)}^2,
\end{align*}
which hold thanks to \eqref{est_1}.
\end{proof}

Next, we prove our theorem regarding convexity of the Dirichlet-to-Neumann maps for \eqref{pf_schrodinger}.

\begin{proof}[Proof of Theorem~\ref{thm_convex}]
Given any $t \in [0,1]$ and $f \in H^{\frac{3}{2}}(\p M)$ let us define
\begin{equation}\label{H_def} H_f(t) = \int_{\p M} f\, \Lambda_{g,(1-t)V_1+tV_2}(f)\,dV_g.\end{equation}
In order to prove \eqref{claim_convexity}, it suffices to prove that
\begin{equation}
\label{claim_ineq}
H''_f(t) \leq 0 \quad \forall\,t\in [0,1] \quad \forall\, f\in H^{\frac{3}{2}}(\p M).
\end{equation}
Let us fix $f \in H^{\frac{3}{2}}(\p M)$. Consider the equation
   \begin{equation}
   \label{t_eq}
   \begin{aligned}
- \Delta_g u_t + ((1-t)\,V_1+t\,V_2) u_t &= 0,\qquad \text{on $M^{\textrm{int}}$},
\\
u_t|_{\p M} &= f.
\end{aligned}
\end{equation}
Defining 
$$ q:= V_2-V_1 \quad \text{on $M^{\textrm{int}}$},$$
we write
    \begin{align*}
u_t = u_0 + t\,v + t^2\,r_t,\qquad \forall\, t\in [0,1]
    \end{align*}
where
 \begin{align}\label{eq_u_0}
- \Delta_g u_0 + V_1 u_0 &= 0,\qquad \text{on $M^{\textrm{int}}$}
\\\notag
u_0|_{\p M} &= f,
    \end{align}
and
    \begin{align}\label{eq_v}
- \Delta_g v + V_1 v+q u_0 &= 0,\qquad \text{on $M^{\textrm{int}}$},
\\\notag
v|_{\p M} &= 0,
    \end{align}
    and finally
    \begin{align}\label{eq_r}
    (-\Delta_g + V_1+ t\,q) r_t+ q\,v &= 0,\qquad \text{on $M^{\textrm{int}}$},
    \\\notag
    r_t|_{\p M} &= 0.
    \end{align}  
    It is straightforward to show that $r_t$ depends smoothly on $t\in [0,1]$ taking values in $H^2(M)$. Define 
    $$ \dot{r}_t = \p_t r_t \quad \text{and}\quad \ddot{r}_t=\p_t^2 r_t\qquad \forall\, t\in [0,1].$$
    We have
\begin{align}\label{eq_r_dot}
(-\Delta_g +V_1+ t\,q) \dot{r}_t+ q\,r_t &= 0, \qquad \text{on $M^{\textrm{int}}$,}
\\\notag
\dot{r}_t|_{\p M} &= 0,
\end{align}  
    and
    \begin{align}\label{eq_r_ddot}
    (-\Delta_g +V_1+ t\,q) \ddot{r}_t+ 2\,q\,\dot{r}_t &= 0, \qquad \text{on $M^{\textrm{int}}$},
    \\\notag
    \ddot{r}_t|_{\p M} &= 0.
    \end{align}  
By elliptic estimates for \eqref{eq_r}--\eqref{eq_r_ddot} we have
\begin{equation}
\label{est}
\|r_t\|_{H^2(M)} +\delta^{-1}\|\dot{r}_t\|_{H^2(M)}+\delta^{-2}\|\ddot{r}_t\|_{H^2(M)} \leq C\,\delta\, \|v\|_{L^2(M)},
\end{equation}
for some constant $C>0$ that depends only on $(M,g)$.

We return to the definition \eqref{H_def} and write
    \begin{align*}
H_f(t) 
&= 
\int_{\p M} u_t \p_\nu u_t\,dV_g 
= 
\int_M u_t \Delta_g u_t + \int_M \pair{\nabla u_t, \nabla u_t}\,dV_g
\\&= \int_M (V_1+tq)u_t^2\,dV_g + \int_M \pair{\nabla u_t, \nabla u_t}\,dV_g
\\&= 
\int_M (V_1u_0^2+|\nabla u_0|^2)\,dV_g + t\int_M (2V_1u_0v+qu_0^2+2\pair{\nabla u_0,\nabla v})\,dV_g\\
&+t^2\int_M(V_1v^2+2qu_0v+2V_1u_0r_t+|\nabla v|^2+2\pair{\nabla u_0,\nabla r_t} )\,dV_g \\
&+ t^3\int_M (qv^2+2qu_0r_t+2V_1vr_t+2\pair{\nabla v,\nabla r_t})\,dV_g\\
&+t^4\int_M(2qvr_t+V_1r_t^2+|\nabla r_t|^2)\,dV_g +t^5 \int_M qr_t^2\,dV_g. 
    \end{align*}
Using equations \eqref{eq_u_0}--\eqref{eq_r}, we observe that certain terms on the right hand side may be simplified through integration by parts. Indeed,
    \begin{align*}
\int_M (V_1u_0r_t+\pair{\nabla u_0,\nabla r_t})\,dV_g&=0,
\\
\int_M (qu_0r_t+V_1vr_t+\pair{\nabla v,\nabla r_t})\,dV_g&=0,
    \end{align*}
    and 
    \begin{align*}
    \int_M (2qu_0v+ |\nabla v|^2+V_1v^2)\,dV_g&=-\int_M |\nabla v|^2\,dV_g-\int_M V_1v^2\,dV_g,\\
    \int_M (qvr_t+V_1r_t^2+|\nabla r_t|^2+tqr_t^2)\,dV_g &=0.
   \end{align*} 
Using these four observations, we obtain
   \begin{align*}
H_f(t)&=\int_M (V_1u_0^2+|\nabla u_0|^2)\,dV_g + t\int_M (2V_1u_0v+qu_0^2+2\pair{\nabla u_0,\nabla v})\,dV_g\\
&-t^2\int_M(V_1v^2+|\nabla v|^2)\,dV_g+ t^3\int_M qv^2\,dV_g+t^4\int_Mqvr_t\,dV_g. 
    \end{align*}
By differentiating $H_f(t)$ twice in the variable $t\in [0,1]$, we deduce that
   \begin{align}
   \label{F_hessian_reduced}
H''_f(t)
&=-2\int_M (|\nabla v|^2+V_1v^2)\,dV_g +6t\int_M qv^2 \,dV_g\\
&\notag + 12t^2\int_M qvr_t\,dV_g+8t^3 \int_M qv\dot{r}_t\,dV_g+t^4\int_M qv\ddot{r}_t\,dV_g.
\end{align}
In view of \eqref{est} we have the following bounds
 \begin{align*}
|\int_M V_1v^2 \,dV_g|+|\int_M q v^2|&\leq 3\delta \|v\|_{L^2(M)}^2,\\
|\int_M qvr_t\,dV_g|+|\int_M qv\dot{r}_t\,dV_g|+|\int_M qv\ddot{r}_t\,dV_g| &\leq 6C\delta^2 \|v\|_{L^2(M)}^2.
\end{align*}
Combining these bounds with equation \eqref{F_hessian_reduced}, and using Poincar\'{e}'s inequality, we see that for $\delta>0$ small, depending only on $(M,g)$, there holds 
$$ H''_f(t) \leq -\int_M |\nabla v|^2\,dV_g\leq 0,$$
thus proving \eqref{claim_ineq}. Moreover, it is clear from the latter inequality that if $H''_f(t_0)=0$ for some nontrivial $f \in H^{\frac{3}{2}}(\p M)$ and some $t_0\in [0,1]$, then 
$$ \int_M |\nabla v|^2\,dV_g=0,$$
and thus $v=0$ on $M$. This in turn implies that $(V_1-V_2)u_0=0$ on $M$. As $f$ is nontrivial, it follows that $V_1-V_2=0$ on $M$.
\end{proof}

We move to proofs for the semilinear equations.

\begin{proof}[Proof of Theorem~\ref{thm_main_0}]
For $\varepsilon \in \R$ in a sufficiently small neighborhood of the origin, let us define $u_\varepsilon$ to be the unique small solution to \eqref{pf_nonlinear} with $F(x,u)=V(x)u^m$, subject to the constant Dirichlet data $f=\varepsilon$ on $\p M$. It can be verified analogously to \cite{FO20} that $u_\varepsilon$ depends smoothly on the parameter $\varepsilon$ with values in $C^{2,\alpha}(M)$. Thus, we may define 
$$ v_k=\frac{1}{k!}\frac{\p^k}{\p \varepsilon^k}u_\varepsilon|_{\varepsilon=0}\qquad k=0,1,2,\ldots.$$
As $u_\varepsilon$ solves \eqref{pf_nonlinear} with $F(x,u)=V(x)u^m$, differentiation in $\varepsilon$ yields that 
$$v_0\equiv 0,\qquad v_1\equiv 1,$$
and that the first nontrivial function $v_k$ with $k\geq 2$ is $v_m\in C^{2,\alpha}(M)$ that solves the elliptic boundary value problem
\begin{equation}\label{v_m_eq}
\begin{aligned}
\begin{cases}
-\Delta_{g}v_m+V=0, 
&\text{on  $M^{\textrm{int}}$},
\\
v_m=0 & \text{on $\p M$}.
\end{cases}
\end{aligned}
\end{equation}
We can therefore write
\begin{equation}\label{u_exp}
u_\varepsilon = \varepsilon + \varepsilon^m v_m(x)+ r_\varepsilon(x)\qquad \forall\, x\in M,
\end{equation}
where
$$ \|r_\varepsilon\|_{C^{2,\alpha}(M)}\leq C |\varepsilon|^{m+1},$$
for some constant $C>0$ independent of $|\varepsilon|$. For future reference, we record that
\begin{equation}
\label{v_m_iden}
\int_M \langle\nabla v_m,\nabla v_m\rangle\,dV_g +\int_M V\,v_m\,dV_g=0,
\end{equation}
a consequence of \eqref{v_m_eq} and Green's identity.

Note that in view of the hypothesis of Theorem~\ref{thm_main_0}, there holds: 
$$\p_\nu u_\varepsilon|_{\p M}\equiv 0.$$
Applying Green's identity we write 
\begin{multline}\label{lhs}
0=\int_{\p M} u_\varepsilon\, \p_\nu u_\epsilon\,dV_g  = \int_{M} u_\varepsilon\, \Delta_g u_\varepsilon\,dV_g + \int_M \langle\nabla u_\varepsilon,\nabla u_\varepsilon\rangle\,dV_g\\
=\int_{M} V\,u_\varepsilon^{m+1}\,dV_g + \int_M \langle\nabla u_\varepsilon,\nabla u_\varepsilon\rangle\,dV_g.
\end{multline}
Using the expansion \eqref{u_exp}, the right hand side of the latter equation can be reduced to the following form:
\begin{multline}\label{rhs}
\varepsilon^{m+1}\,\int_{M}V\,dV_g + (m+1)\varepsilon^{2m}\int_M V\,v_m\,dV_g + \\
	\varepsilon^{2m}\int_M \langle\nabla v_m,\nabla v_m\rangle\,dV_g + O(|\varepsilon|^{2m+1}).
	\end{multline}
It follows from \eqref{lhs} and \eqref{rhs} that there holds:
$$ \int_{M}V\,dV_g=0,$$
and also that
$$
\int_M \langle\nabla v_m,\nabla v_m\rangle\,dV_g + (m+1)\int_M V\,v_m\,dV_g=0.
$$
Recalling \eqref{v_m_iden}, the latter expression reduces to
\begin{equation}
\label{key_iden}
\int_M \langle\nabla v_m,\nabla v_m\rangle\,dV_g=0.
\end{equation}
It follows from \eqref{key_iden} together with the fact that $v_m|_{\p M}=0$ that $v_m$ is identically zero. Finally, it follows from \eqref{v_m_eq} that $V$ is also identically zero.
\end{proof}

Finally, we prove the theorem regarding recovery of general Riemannian manifolds.

\begin{proof}[Proof of Theorem~\ref{thm_main}]

Recall that $F(x,s)=s\,G(s)$ with $s \in \R$ where $G$ is a smooth function that satisfies (H1). As a consequence $F$ satisfies (A1)--(A2) with $J=\N$. Given any $k \in \N$, there are constants $\delta_{k,0},\delta_{k,1}>0$ depending only on $M$ , $g_1$, $g_2$  and $k$ such that for any $f\in C^{\infty}(\p M)$ subject to 
\begin{equation}\label{f_small}
\|f\|_{C^{2,\alpha}(\p M)}\leq \delta_{k,0},\end{equation}
and given $j=1,2$, there exists a unique solution $$u_j:=L_{g_j,k}f\in C^{\infty}(M)$$ 
to the semilinear equation
\begin{equation}\label{nonlinear_eq_1}
\begin{cases}
-\Delta_{g_j}u_j+u_j\,G(u_j)=0 & \mbox{on}\ M^{\textrm {int}},\\  
u_j=t_k+f & \mbox{on}\ \p M.
\end{cases}
\end{equation}
subject to the constraint $\|u_j-t_k\|_{C^{2,\alpha}(M)}\leq \delta_{k,1}$. It follows that given any $f\in C^{\infty}(\p M)$ there holds 
$$ \frac{\p}{\p\epsilon} \Lambda_{g_j,F}(t_k+\epsilon f)|_{\epsilon=0}= \p_{\nu_j} v^{(j)}_{k}|_{\p M}=:\Gamma_{g_j,k}(f),\quad j=1,2,$$
where $\nu_j$ is the exterior unit normal field on $\p M$ with respect to the metric $g_j$ and $v^{(j)}_{k}\in C^{\infty}(M)$, $j=1,2,$ $k\in \N$, is the unique solution to the linearized equation
 \begin{equation}\label{linear_eq}
\begin{cases}
-\Delta_{g_j}v^{(j)}_{k}+t_k\,G'(t_k)\,v^{(j)}_{k}=0 & \mbox{on}\ M^{\textrm {int}},\\  
 v^{(j)}_{k}=f & \mbox{on}\ \p M.
\end{cases}
 \end{equation}
 We remark that for each $k\in \N$, the map $\Gamma_{g_j,k}$, $j=1,2,$ is the Dirichlet-to-Neumann map associated to the linear equation \eqref{linear_eq}. Observe also that $\Lambda_{g_j,F}$ determines the maps $\Gamma_{g_j,k}$ for any $k\in \N$.  By density of $C^{\infty}(\p M)$ in $H^{\frac{3}{2}}(\p M)$, it follows that 
\begin{equation}\label{Gamma_eq} \Gamma_{g_1,k}(f)=\Gamma_{g_2,k}(f)
\end{equation}
for all $f\in H^{\frac{3}{2}}(\p M)$ and all $k\in \N$.  

Let $j=1,2$ and let us denote by $0<\lambda^{(j)}_1\leq \lambda_2^{(j)}\leq \ldots$ the Dirichlet eigenvalues of $-\Delta_{g_j}$ on $M$ and by $\{\phi_\ell^{(j)}\}_{\ell=1}^{\infty}\subset H^1_0(M)$ an $L^2(M)$-orthonormal basis (with respect to the metric $g_j$) such that 
  \begin{equation}\label{eigen_eq}
\begin{cases}
-\Delta_{g_j}\phi_{\ell}^{(j)}=\lambda_\ell^{(j)}\phi_\ell^{(j)} & \mbox{on}\ M^{\textrm {int}},\\  
 \phi_{\ell}^{(j)}=0 & \mbox{on}\ \p M.
\end{cases}
 \end{equation}
For each $f \in H^{\frac{3}{2}}(\p M)$ and each
$$ z \in \mathbb H = \{x+\textrm{i}\,y\in \C\,:\, x\geq 0\},$$
consider the equation
\begin{equation}\label{linear_eq_2}
\begin{cases}
(-\Delta_{g_j}+z)u^{(j)}_{z}=0 & \mbox{on}\ M^{\textrm {int}},\\  
u_{z}^{(j)}=f & \mbox{on}\ \p M.
\end{cases}
\end{equation}
Writing $\mu=t_1G'(t_1)$, $\psi_k^{(j)}= \p_{\nu_j} \phi^{(j)}_k|_{\p M}$, $j=1,2,$ and $k\in\N$, we define for each $f \in H^{\frac{3}{2}}(\p M)$ and $j=1,2,$ the partial sums
$$ K_{f,N}^{(j)}(z,x) = \sum_{k=1}^{N} \frac{(f,\psi^{(j)}_k)_{L^2(\p M)}}{(\lambda^{(j)}_k+z)(\lambda_k^{(j)}+\mu)}\phi_k^{(j)}(x)\in H^1_0(M)\cap C^{\infty}(M).$$
Using Green's identity we note that 
$$(f,\psi_k^{(j)})_{L^2(\p M)}= -\lambda_k^{(j)}\int_M u_0^{(j)}\,\phi_k^{(j)}\,dV_{g_j}\quad j=1,2,\quad k\in \N,$$  where $u_0^{(j)}$ is the unique solution to \eqref{linear_eq_2} with $z=0$ and Dirichlet data $f$. Returning to the partial sums, we obtain that
$$ \Delta_{g_j}K_{f,N}^{(j)}=\sum_{k=1}^{N} (\lambda_k^{(j)})^2\,\frac{(u^{(j)}_0,\phi^{(j)}_k)_{L^2( M)}}{(\lambda^{(j)}_k+z)(\lambda_k^{(j)}+\mu)}\phi_k^{(j)}$$
As $\{\phi_k^{(j)}\}_{k=1}^{\infty}\subset H^1_0(M)$ is an $L^2(M)$-orthonormal basis and as $\mu>0$ and $z \in \mathbb H$, we deduce that 
$$ \|\Delta_{g_j}K_{f,N}^{(j)}\|_{L^2(M)} \leq \|u_0^{(j)}\|_{L^2(M)}\leq C\|f\|_{H^{\frac{3}{2}}(\p M)},$$
for some constant $C>0$ that is independent of $z\in \mathbb H$, $j=1,2,$ and $N\in \N$. By elliptic regularity, we deduce that 
$$ \|K_{f,N}^{(j)}\|_{H^2(M)}\leq C\|f\|_{H^{\frac{3}{2}}(\p M)},$$
for some constant $C>0$ that is independent of $z\in \mathbb H$, $j=1,2,$ and $N\in \N$. As a consequence of the latter computation and continuity of trace operators, we may define the Banach valued holomorphic function
$$\zeta_f^{(j)}(z)= \sum_{k=1}^{\infty} \frac{(f,\psi^{(j)}_k)_{L^2(\p M)}}{(\lambda^{(j)}_k+z)(\lambda_k^{(j)}+\mu)}\psi_k^{(j)}, \quad \text{in $H^{\frac{1}{2}}(\p M)$},$$
for all $z\in \mathbb H$ and note that $\zeta_f^{(j)}$ is uniformly bounded on $\mathbb H$ in the sense of taking values in $H^{\frac{1}{2}}(\p M)$. Following \cite[Lemma 2.2]{eric} with trivial modifications, we record the identity
\begin{equation}\label{normal_diff}
\p_{\nu_j} u^{(j)}_z|_{\p M} - \p_{\nu_j} u^{(j)}_{\mu}|_{\p M} = (z-\mu)\,\zeta_f^{(j)}(z),\quad j=1,2\quad z \in \mathbb D,\end{equation}
where $u^{(j)}_z$ is the unique solution to \eqref{linear_eq_2} with Dirichlet data $f$ and $u^{(j)}_\mu$ is the unique solution to \eqref{linear_eq_2} with $z=\mu=t_1G'(t_1)$ and with Dirichlet data $f$.

For each fixed $f\in H^{\frac{3}{2}}(\p M)$ let us define $\zeta_f: \mathbb H \to H^{\frac{1}{2}}(\p M)$ via
\begin{equation}
\label{zeta_def}
\zeta_f(z)= \zeta_f^{(1)}(z)-\zeta_f^{(2)}(z).\end{equation}
In view of the above discussion, we have that $\zeta_f(z)$ is holomorphic and uniformly bounded on $z\in \mathbb H$ in the sense of taking values in $H^{\frac{1}{2}}(\p M)$. Moreover, we have
$$ \zeta_f(t_kG'(t_k))=0 \quad \text{for $k=2,3,\ldots$}$$ 
thanks to \eqref{Gamma_eq}, and \eqref{normal_diff}--\eqref{zeta_def}.

Let $\mathbb D$ denote the open unit disk in the complex plane and for each $z\in \mathbb D$, define $h_f(z)= \zeta_f(\frac{1+z}{1-z})$. It is clear that $h_f$ is a bounded holomorphic function on $\mathbb D$ and as a consequence either $h_f\equiv 0$ or else its nonzero roots must satisfy the well known Blaschke condition. Noting that $h_f(r_k)=0$ for $r_k=\frac{t_kG'(t_k)-1}{1+t_kG'(t_k)}$, $k=2,3,\ldots$, and recalling that $t_kG'(t_k)>1$, we write
$$ \sum_{k=2}^{\infty} (1-|r_k|) = 2\sum_{k=2}^{\infty}\frac{1}{1+t_kG'(t_k)} =\infty,$$
where in the last step we used (H1). Therefore, the Blaschke condition is not satisfied for the nonzero roots of $h_f(z)$, implying that $h_f(z)\equiv 0$ on $\mathbb D$, and thus    
$$\zeta_f(z)=0,\quad \forall\, z\in \mathbb H, \quad f\in H^{\frac{3}{2}}(\p M).$$
The latter equality implies that the Dirichlet-to-Neumann maps for the equation \eqref{linear_eq_2} with $j=1,2$ are equal to each other for all $z \in \mathbb H$. Applying \cite[Theorem 1]{KKLM}, it follows that the Dirichlet spectral data for the two manifolds $(M,g_1)$ and $(M,g_2)$ are identical and as such we have reduced our inverse problem to the standard Gel'fand spectral inverse problem studied in \cite{BK92}. Applying the main result in \cite{BK92}, we conclude that there is a diffeomorphism 
$$\Phi: M \to M,$$
that is equal to identity on $\p M$ and such that
$$\Phi^\star g_2=g_1.$$
\end{proof}

\end{document}